\documentclass[11pt]{article}
\usepackage{xcolor}
\usepackage{amsmath}
\usepackage{amsfonts}
\usepackage{amsthm}
\usepackage{enumerate}
\usepackage[english]{babel}
\usepackage{calrsfs}
\newcommand{\FF}{\mathbb F}

\newcommand{\cM}{\mathcal M}

\newcommand{\cP}{\mathcal P}
\newcommand{\cV}{\mathcal V}

\newcommand{\cS}{\mathcal S}

\newcommand{\cL}{\mathcal L}
\newcommand{\cG}{\mathcal G}

\newcommand{\cC}{\mathcal C}

\newcommand{\PG}{\mathrm{PG}}

\newcommand{\GL}{\mathrm{GL}}

\newcommand{\fB}{\mathfrak B}

\newcommand{\cI}{\mathcal{I}}

\newtheorem{theorem}{Theorem}[section]
\newtheorem{lemma}[theorem]{Lemma}
\newtheorem{corollary}[theorem]{Corollary}

\newtheorem{conjecture}[theorem]{Conjecture}
\theoremstyle{definition}
\newtheorem{remark}[theorem]{Remark}
\newtheorem{definition}[theorem]{Definition}

\makeatletter
\let\ltxxlabel\ltx@label
\makeatother

\title{Grassmannians of codes}
\author{I. Cardinali, L. Giuzzi}
\begin{document}

\maketitle

\begin{abstract}
	\noindent Consider the point line-geometry $\cP_t(n,k)$ having as points all the $[n,k]$-linear codes having minimum dual distance at least $t+1$ and where two points $X$ and $Y$ are collinear whenever $X\cap Y$ is a $[n,k-1]$-linear code having minimum dual distance at least $t+1$.
	We are interested in the collinearity graph $\Lambda_t(n,k)$ of $\cP_t(n,k).$ The graph $\Lambda_t(n,k)$ is a subgraph of the Grassmann graph and also a subgraph of the graph $\Delta_t(n,k)$ of the linear codes having minimum dual distance at least $t+1$ introduced in~\cite{KP16}.
	We shall study the structure of $\Lambda_t(n,k)$ in relation to
        that of $\Delta_t(n,k)$ and we will characterize the set of its isolated vertices. We will then focus on $\Lambda_1(n,k)$  and $\Lambda_2(n,k)$ providing necessary and sufficient conditions for them
	to be connected.
\end{abstract}

\leftline{{\bfseries MSC:} 51E22, 94B27 }
\leftline{{\bfseries Keywords:} Grassmann graph, Linear codes, Point-Line
  geometry, Connectivity}

\section{Introduction}
Let $V:=V(n,q)$ be an $n$-dimensional  vector space
over the finite field $\FF_q$ with $q$ elements
and $0<k<n$; let $\Gamma(n,k)$ be the graph
whose vertices are the $k$-subspaces of $V$ and where two
vertices $X,Y$ are adjacent if and only if
$\dim(X\cap Y)=k-1$. The graph $\Gamma(n,k)$ is called the $k$-Grassmann graph of $V.$

An $[n,k]$-linear code is just a vertex of $\Gamma(n,k)$. We say that
a $[n,k]$-linear code $C$ has dual minimum distance at least $t+1$ if, given a
generator matrix $G$ for $C$, any set of $t$ columns of $G$ is linearly
independent. It is easy to see that this condition does not depend on the
particular generator matrix chosen for $C$ and that
this is the same as to require that the dual code of $C$
has minimum Hamming distance at least $t+1$, whence the name.

Denote by $\cC_t(n,k)$ the set of all $[n,k]$-codes with dual minimum
distance at least $t+1$ over $\FF_q$.
By construction,
the elements of $\cC_t(n,k)$ can be regarded as the vertices of the induced
subgraph $\Delta_t(n,k)$ of $\Gamma(n,k)$,
where two codes  $X$, $Y$, considered as two $k$-dimensional subspaces, are adjacent ($X\sim_{\Delta} Y$) if and only
if $\dim(X\cap Y)=k-1$.

The graph $\Delta_t(n,k)$ has been called the {\it{Grassmann graph of the linear codes with dual minimum distance at least $t+1$}}.
Clearly, $0\leq t\leq k.$
For $t=0$, $\cC_0(n,k)$ consists of the class of all $k$-subspaces
of $V$ and $\Delta_0(n,k)=\Gamma(n,k)$; 
for $t=1$, $\cC_1(n,k)$ is called the class of the \emph{non-degenerate $[n,k]$-linear codes} and for $t=2$, $\cC_2(n,k)$ is called the class of the \emph{projective $[n,k]$-linear codes}.
The case $t=k$
corresponds to codes whose dual is MDS (\emph{maximum distance separable}).
In particular as the duals of MDS codes are in turn MDS (see \cite{MS}), the
elements of $\cC_t(n,k)$ are themselves MDS.

The graphs $\Delta_t(n,k)$ for $t=1$ (non-degenerate codes) and for $t=2$
(projective codes) have been studied respectively in \cite{KP16} and
\cite{KPP18}. The graph $\Delta_t(n,k)$ for $t\geq3$ has been introduced and
studied in~\cite{ILK21}.


The following theorem synthetically reports on the main results
of interest here regarding the graph $\Delta_t(n,k).$
\begin{theorem}[\cite{ILK21,KP16,KPP18}]
	\label{told}
	\begin{enumerate}
		\item $\Delta_1(n,k)$ is connected for any $q$; furthermore $\Delta_1(n,k)$  is isometrically embedded in the $k$-Grassmann graph $\Gamma(n,k)$ if and only if  $n<(q+1)^2+k-2$.
		\item If $q\geq{n\choose 2}$ then $\Delta_2(n,k)$ is connected and it is
		      isometrically embedded in the $k$-Grassmann graph $\Gamma(n,k)$; furthermore, 
		      $\Delta_2(n,k)$ and $\Gamma(n,k)$ have the same diameter.
		\item If $t>2$ and
		      $q\geq{n\choose t}$ then $\Delta_t(n,k)$ is connected and it is isometrically embedded in the $k$-Grassmann graph
		      $\Gamma(n,k)$; furthermore, $\Delta_t(n,k)$ and $\Gamma (n,k)$ have
		      the same diameter.
	\end{enumerate}
\end{theorem}

In this paper, we shall consider a subgraph of $\Delta_t(n,k)$, which we will denote by $\Lambda_t(n,k)$, defined as follows: the vertices of  $\Lambda_t(n,k)$ are all the elements of $\cC_t(n,k)$ and two vertices $X$ and $Y$ are adjacent in  $\Lambda_t(n,k)$ ($X\sim_{\Lambda} Y$) whenever their intersection $X\cap Y$ belongs to $\cC_t(n,k-1)$.
In other words, the vertices of $\Lambda_t(n,k)$
are the same as the vertices of $\Delta_t(n,k)$, but the
condition for an edge to exist is stronger.

Consider now the point-line geometry $\cP_t(n,k):=(\cC_t(n,k), \cL_t(n,k))$ where the points
are the elements of $\cC_t(n,k)$
and the lines are defined as:
\begin{itemize}
	\item $\ell_{X,Y}:=\{Z\in \cC_t(n,k): X\subset Z\subset Y\}$ with
	      $X\in\cC_t(n,k-1)$, $Y\in\cC_t(n,k+1)$ if $k<n-1$;
	\item $\ell_X:=\{Z\in \cC_t(n,n-1): X\subset Z\}$ with
	      $X\in\cC_t(n,n-2)$ if $k=n-1$.
\end{itemize}

Observe that if $Z$ is a $k$- dimensional vector subspace and it contains a subspace $D\in\cC_t(n,k')$
with $k'< k$, then $Z\in\cC_t(n,k)$. In particular, the line
$\ell_{X,Y}\in \cL_t(n,k)$ can also be described as
$\ell_{X,Y}:=\{Z: X\subset Z\subset Y\}$ with $X\in\cC_t(n,k-1)$
and $\dim(Y)=k+1$.

The geometry $\cP_t(n,k)$ is a subgeometry of the \emph{$k$-Grassmann geometry} $\cG(n,k)$ (see Section~\ref{Sec2}) and the collinearity graph of $\cP_t(n,k)$ is precisely the graph $\Lambda_t(n,k)$.

In this paper,
we shall study the structure of $\Lambda_t(n,k)$ and the interplay between the geometry $\cP_t(n,k)$ and the geometry $\cG(n,k).$
We point out that from a more applied point of view, the study of the graph $\Lambda_t(n,k)$ is related to some code density problem, but we leave
its investigation to further works.

Before stating our main results we need to give the following definitions.

\begin{definition}\label{isolated code}
	A code $C\in\cC_t({n,k})$ is \emph{isolated} if $C$ does
	not contain any proper subcode $D\in\cC_t({n,k-1})$.
	We denote the set of all isolated codes in $\cC_t(n,k)$ by
	the symbol $\cI_t(n,k)$.
\end{definition}
If a code $C$ is not isolated, then
there exists at least one code
$C'\in\cC_t(n,k)$ with $C'\neq C$
and such that $C\cap C'\in\cC_t(n,k-1)$, that is $C'\sim_\Lambda C$.
So, it follows readily that
a vertex of $\Lambda_t(n,k)$ is {\it isolated} if and only if it corresponds to an isolated code of $\cC_t({n,k}).$
On the other hand, the graph induced by $\Lambda_t(n,k)$ on $\cI_t(n,k)$ is \emph{totally disconnected}, i.e. it contains no edge. We will use the same symbol $\cI_t(n,k)$ to denote both the set of isolated codes of $\cP_t(n,k)$ and the subgraph $(\cI_t(n,k),\emptyset)$
of isolated vertices of $\Lambda_t(n,k).$

Given a vector space of dimension $k$ over $\FF_q$, we denote by $\PG(k-1,q)$ the associated projective space; if $X$ is a set of points of $\PG(k-1,q)$, then $\langle X\rangle$ denotes the projective subspace spanned by $X$.
The graph $\Lambda_t(n,k)$ is related to interesting and well studied configurations of points of a projective space such as the {\it $t$-saturating sets} whose definition we report below.

\begin{definition}\label{saturated}
	Let $t$ be an integer $0\leq t\leq k$.
	For any $\Omega\subseteq\PG(k-1,q)$
	let
	\[ \cS_t(\Omega):=\bigcup_{\begin{subarray}{c}
				X\subseteq\Omega \\
				|X|=t+1
			\end{subarray}}\langle X\rangle \]
	be the set of all points of $\PG(k-1,q)$ on subspaces
	spanned by $t+1$ points of $\Omega$.
	The set $\Omega$ is \emph{$t$-saturating} if $\cS_t(\Omega)=\PG(k-1,q)$ and
	$\cS_{t-1}(\Omega)\neq\PG(k-1,q)$.
\end{definition}
If $t=0$, then the only \emph{$0$-saturating} set of $\PG(k-1,q)$ is the point set of $\PG(k-1,q).$
If $t=1$, then a \emph{$1$-saturating} of $\PG(k-1,q)$ is a set $\Omega$
whose $2$-secants cover $\PG(k-1,q)$, i.e. for any $P\in\PG(k-1,q)$
there are $X,Y\in\Omega$ such that $P\in\langle X,Y\rangle$.

We refer to~\cite{Dav03,Dav19,G13} and the references therein for more information on saturating sets and for the known bounds on the minimal cardinality of a $t$-saturating set and related constructions.

\subsection{Main Results}
The following are the main results of this paper.
\begin{theorem} \label{main thm 1}
	The graph
	$\Delta_t(n,k-1)$ is connected if and only if
	the subgraph $\widetilde{\Lambda}_t(n,k)$ of
	$\Lambda_t(n,k)$ induced by all non-isolated codes
	in $\cC_t(n,k)\setminus\cI_t(n,k)$ is connected.
	In particular, $\Lambda_t(n,k)=\widetilde{\Lambda}_t(n,k)\cup\cI_t(n,k)$.
\end{theorem}
\begin{remark}
	If $\cC_t(n,k-1)=\emptyset$, then $\Delta_t(n,k-1)$ is an empty
	graph and $\Lambda_t(n,k)=\cI_t(n,k)$ consists only of isolated
	vertices.
	If we take by convention the empty graph to be connected, then
	Theorem~\ref{main thm 1} is true also in this case.
	Also the converse holds, i.e. if
	$\Lambda_t(n,k)=\cI_t(n,k)$, then $\cC_t(n,k-1)=\emptyset$. By Theorem~\ref{large-q}, this holds for $t=k.$
\end{remark}

\begin{theorem}\label{main thm 2}
	A code $C\in\cC_t(n,k)$ is isolated if and only if the columns of any
	generator matrix of $C$ are vector representatives of a $(t-1)$-saturating set
	of $\PG(k-1,q)$.
\end{theorem}

The following corollaries, holding for $t=1$ and $t=2$, are consequences of~Theorem~\ref{told}, Theorem~\ref{main thm 1} and Theorem~\ref{main thm 2}.
For Point 2. of Corollary~\ref{t=1}, see also \cite[Proposition 4]{KP17}. 

\begin{corollary}
	\label{t=1}
	The following hold.
	\begin{enumerate}
		\item $\Lambda_1(n,k)=\widetilde{\Lambda}_1(n,k)\cup \cI_1(n,k)$ where $\widetilde{\Lambda}_1(n,k)$ is connected.
		\item
		      The graph $\Lambda_1(n,k)$ is connected if and
		      only if $n<\frac{q^k-1}{q-1}$.
		\item A code $C\in{\cC}_1(n,k)$ is isolated if and only if the columns of
		      any generator matrix of $C$ are vector representatives of all the points of a $(k-1)$-dimensional projective space.
	\end{enumerate}
\end{corollary}

\begin{corollary}
	\label{t=2}
	The following hold.
	\begin{enumerate}
		\item 	$\Lambda_2(n,k)=\widetilde{\Lambda}_2(n,k)\cup \cI_2(n,k)$
		      where $\widetilde{\Lambda}_2(n,k)$ is connected.
		\item The graph $\Lambda_2(n,k)$ is connected if and only if
		      \[n<\min\{|\Omega|\colon \Omega \text{ is a $1$-saturating set of }
			      \PG(k-1,q)\}.\]
		\item A code $C\in\cC_2(n,k)$ is isolated if and only if the $2$-secants of the projective set determined by the columns of any generator matrix of $C$ cover all the points of a $(k-1)$-dimensional projective space.
	\end{enumerate}
\end{corollary}

\begin{theorem}
	\label{large-q}
	If
	$q>{n\choose t}$ and $t<k$, then  the graph $\Lambda_t(n,k)$ is
	connected.
	For $t=k$, the graph $\Lambda_{k}(n,k)$ is totally disconnected.
\end{theorem}

In light of Theorem~\ref{large-q} it makes sense to ask for given
$t$, $k$ and $q$ what is the structure of the graph $\Lambda_t(n,k)$, as
$n$ grows, under the condition $\cC_t(n,k)\neq\emptyset$.
Keeping in mind Definition~\ref{isolated code} of isolated code,
we can consider the following two parameters
\[ \nu_t(k;q):=\min\{ n: \Lambda_t(n,k)\text{ is disconnected}\}, \]
\[ \nu_t^+(k;q):=\min\{ n: \cI_t(n,k)\neq\emptyset\}. \]
If a graph has an isolated vertex then it is disconnected, so  we have
\[ \nu_t(k;q)\leq\nu_t^+(k;q).\]
Clearly, if $\nu_t(k;q)=\nu_t^+(k;q)$ then the graph $\Lambda_t(n,k)$ is disconnected if and only if it has isolated vertices.
We conjecture the following.
\begin{conjecture}
	\label{conj}
	For any $q$, $k$ and $t$ with $k\geq t$,
	$\nu_t(k;q)=\nu_t^+(k;q)$.
\end{conjecture}
We observe that Conjecture~\ref{conj} holds true for $t=1$ by Corollary~\ref{t=1} and for $t=2$ by Corollary~\ref{t=2}.

\par\medskip
Moving to the geometry $\cP_t(n,k)$ whose collinearity graph is
$\Lambda_t(n,k)$,
we shall prove that the inclusion map $\iota$ of
$\cP_{t}(n,k)$ in the Grassmann geometry $ \cG(n,k)$ has the property that
lines of $\cP_t(n,k)$ are mapped into lines of $\cG(n,k)$ and
the preimage of any line of $\cG(n,k)$ contained
in the image of $\iota$ is still a line of $\cP_t(n,k).$ According to~\cite{transp}, this is exactly the definition for an embedding to be {\it transparent} (see Section~\ref{transparenza}).

\begin{theorem}
	\label{t0}
	The inclusion map $\iota:\cP_{t}(n,k)\to\cG(n,k)$ is a transparent embedding.
\end{theorem}
Since $\cG(n,k)$ is {\it projectively embeddable} by means of the Pl\"ucker embedding
$\varepsilon_k$ in the projective space $\PG(\bigwedge^kV)$, Theorem~\ref{t0}  implies that
the geometry $\cP_t(n,k)$ is also projectively embeddable in $\PG(\bigwedge^k V)$ by
means of the restriction of $\varepsilon_k$ to $\cP_t(n,k),$ and
this embedding is transparent.

The resulting point set $\cV_t:=\varepsilon_k(\cC_t(n,k))$ is a subset of the Grassmann variety in $\PG(\bigwedge^k V)$ and,
by Theorem~\ref{t0} and the transparency of the Pl\"ucker
embedding $\varepsilon_k,$
it has the property that
for any two distinct points $P,Q\in\cV_t,$ the projective line joining
$P$ and $Q$ is fully contained in $\cV_t$ if and only
if $\varepsilon_k^{-1}(P)$ and $\varepsilon_k^{-1}(Q)$ are
adjacent in $\Lambda_t(n,k)$.

In particular, it is possible to reconstruct the graph $\Lambda_t(n,k)$
from just looking at the point set $\cV_t$.
We observe that this property does not hold for the graph
$\Delta_t(n,k)$, i.e. it can happen that $X$ and $Y$ are adjacent in
$\Delta_t(n,k)$ but the projective line $\ell$ joining
$\varepsilon_k(X)$ and $\varepsilon_k(Y)$ is not fully contained in
$\cV_t$, i.e.
there exists a point $Z$ on $\ell$ such that
$\varepsilon^{-1}(Z)\not\in \cC_t(n,k).$




\bigskip
\noindent {\textbf{Structure of the paper}}
In Section~\ref{Sec2} we will set the notation and recall the basic results regarding the objects of our interest. In Section~\ref{general properties} we shall focus on the graph $\Lambda_t(n,k)$ proving some general results for arbitrary values of $t$,$k$ and $n.$
In Section~\ref{sec3} we will prove our main theorems; in particular we will prove Corollary~\ref{t=1} in Subsection~\ref{non-deg codes} and Corollary~\ref{t=2} in Subsection~\ref{projective codes}. The general case is treated in Subsection~\ref{general codes} and Theorem~\ref{t0} is proved in Section~\ref{transparenza}.

\section{Preliminaries}\label{Sec2}
As in Introduction, let $V:=V(n,q)$ be a $n$-dimensional vector space over
the finite field $\FF_q$ and $0<k<n$.
The  \emph{$k$-Grassmann geometry} is defined as the point-line
geometry $\cG(n,k)$ whose points are the $k$-dimensional subspaces
of $V$ and whose lines are the sets
\begin{itemize}
	\item $\ell_{X,Y}:=\{Z: X\subset Z\subset Y\}$ with
	      $\dim(X)=k-1$, $\dim(Y)=k+1$ if $k<n-1$;
	\item $\ell_X:=\{Z: X\subset Z\}$ with
	      $\dim(X)=n-2$ if $k=n-1$.
\end{itemize}

This geometry has been widely investigated (see for example~\cite{S11}) and its collinearity graph is the Grassmann graph
$\Gamma(n,k)$ as defined in Introduction.

Suppose that $\fB=({\mathbf{e}}_1,\dots,{\mathbf{e}}_n)$ is a given fixed
basis of $V$; henceforth we shall always
write the coordinates of the vectors in $V$ with respect to $\fB$.

Given two vectors $\mathbf{x}=\sum_{i=1}^n{x_i}{\mathbf e}_i$ and $\mathbf{y}=\sum_{i=1}^n{y_i}{\mathbf e}_i$,
the Hamming distance (with respect to the basis $\fB$) between $\mathbf{x}$ and $\mathbf{y}$
is $d(\mathbf{x},\mathbf{y}):=|\{i: x_i\neq y_i\}|.$
In this setting, a \emph{$[n,k]$-linear code} $C$ is just
a $k$-dimensional vector subspace of $V$ together with the restriction to
$C\times C$ of the Hamming distance induced by $\fB$.

If $B_C$  is an ordered  basis of $C$,  a \emph{generator matrix} $G_C$
for $C$ is the $k\times n$ matrix whose rows are the components
of the elements of $B_C$ with respect to $\fB$.
Given a $[n,k]$-linear code $C$, its \emph{dual code}
is the $[n,n-k]$-linear code $C^{\perp}$ given by
\[ C^{\perp}:=\{ \mathbf{v}\in V: \forall \mathbf{c}\in C, \mathbf{v}\cdot\mathbf{c}=0 \} \]
where $\cdot$ denotes the standard symmetric bilinear form on $V$
given by
\[ (v_1{\mathbf e}_1+\dots+v_n{\mathbf e}_n)\cdot (c_1{\mathbf e}_1+\dots+c_n{\mathbf e}_n)=v_1c_1+\dots+v_nc_n. \]
Since the bilinear form  ``$\cdot$'' is non-degenerate, $C^{\perp\perp}=C$.
We say that $C$ has \emph{dual minimum distance} at least $t+1$ if
and only if the minimum Hamming distance of the dual code $C^{\perp}$ of $C$ is at least $t+1$. As mentioned in  Introduction, this condition is equivalent
to saying that for any generator matrix $G_C$ of $C$, any set of $t$
columns of $G_C$ is linearly independent.

For $t\in \mathbb{N}$ we already defined the set $\cC_t(n,k)$ of all $[n,k]$-linear codes with dual minimum distance at least $t+1$ and the point-line geometry $\cP_t(n,k)=(\cC_{t}(n,k),\cL_t(n,k))$
whose collinearity graph is precisely $\Lambda_t(n,k).$
The geometry $\cP_t(n,k)$, clearly, can be regarded as a subgeometry of $\cG(n,k).$


We recall that a graph $\Gamma$ is \emph{connected} whenever given any two of its vertices (say $P$ and $Q$ with $P\neq Q$)
there exists a path in $\Gamma$, i.e.
a sequence of non-repeated adjacent vertices, starting at $P$ and
ending at $Q$. We consider the empty graph to be connected.
The length of a path between $P$ and $Q$ is defined as the smallest cardinality of a path connecting $P$ and $Q$ diminished by $1$; the distance $d_{\Gamma}(P,Q)$ between $P$ and $Q$ in $\Gamma$ is the length of a shortest path between $P$ and $Q$;  a connected component of a graph is a
maximal
non-empty subset of its vertices such that the subgraph induced on it is connected.
The diameter of a graph is the maximum distance between any two vertices; see~\cite{BM82} for a general reference about graph theory.
In general we say that a subgraph $\Gamma'$ is isometrically embedded in a larger graph $\Gamma$ if there exists a distance-preserving map $\Gamma'\to
\Gamma$; see~\cite{GW85}.
\medskip

We already mentioned in Introduction the subgraph $\Delta_{t}(n,k)$ of the Grassmann graph having as vertices the $[n,k]$-linear codes $\cC_t(n,k)$ with dual minimum distance at least $t+1.$
The graph $\Delta_t(n,k)$ was first introduced in~\cite{KP16} where the authors focused mostly on the case of non-degenerate linear codes, i.e. on $\Delta_1(n,k)$.
In~\cite{KPP18} these results, among others, were extended to the  graph $\Delta_2(n,k)$
of projective linear codes. In~\cite{ILK21} we considered properties of the graph
$\Delta_t(n,k)$ for arbitrary $t$.

Comparing the graph $\Lambda_t(n,k)$ with the graph $\Delta_t(n,k)$, we observe that $\Lambda_t(n,k)$ has the same vertices as
$\Delta_t(n,k)$, but less edges. In particular, there exist codes
$A,B\in\cC_t(n,k)$ such that $A\sim_{\Delta}B$, but $A\not\sim_{\Lambda}B$.
Thus $d_{\Lambda}(A,B)>d_{\Delta}(A,B)$ and $\Lambda_t(n,k)$ is not
isometrically embedded in $\Delta_t(n,k)$; consequently, $\Lambda_t(n,k)$
is not isometrically embedded in the Grassmann graph either.

\subsection{Equivalent codes}
Given the basis $\fB$, the \emph{monomial group} $\cM(V)$ of $V$ consists of all linear
transformations of $V$ which map the set of subspaces
$\{ \langle {\mathbf e}_1\rangle,\dots,\langle {\mathbf e}_n\rangle \}$ in itself.
It is straightforward to see that
${\cM(V)}\cong \FF_q^{\times}\wr S_n$, where $\wr$ denotes the wreath product  and $S_n$ is the symmetric group of order $n$; see~\cite[Chapter 8,\S5]{MS} for more details.

\begin{definition}\label{def equiv}
	Two $[n,k]$-linear codes $X$ and $Y$ are \emph{equivalent}
	if there exists a monomial transformation $\rho\in\cM(V)$ such
	that $X=\rho(Y)$.
\end{definition}

Equivalence between linear codes is an equivalence relation and
the equivalence class of a code $X$ corresponds to the orbit of
$X$ under the action of $\cM(V)$ on the subspaces
of $V$.

A monomial transformation $\rho$ is given by a linear transformation of $V$ which sends the code $Y$ (regarded as a $k$-dimensional subspace of $V$), into the code $X=\rho(Y).$
In particular, if $G_Y$ is a generator matrix for $Y$, then the matrix $\rho(G_Y)$ obtained from $G_Y$ by applying $\rho$ to each of its rows, is a generator matrix for $X$. We put $G_{\rho(Y)}:=\rho(G_Y).$ Observe that the transformation induced by $\rho$ acts on the columns of $G_Y.$
Indeed,  if $\rho$ is represented by a $n\times n$ matrix $R$ with respect to the
basis $\fB$, then $G_{\rho(Y)}=G_Y\cdot R.$

The matrix $R$ factors as a product $R=PD$, where $P$ is
a permutation matrix and $D$ is a non-singular diagonal matrix.
In the rest of this paper, we shall denote, with a slight abuse of notation,
by the same symbol $\rho$ not only the isometry $\rho:V\to V$, but
also the corresponding map acting on the generator matrices of the codes.

If $X$ is a $[n,k]$-linear code  with generator matrix
$G_X$ and $A\in\GL(k,\FF_q)$, then $G'_X=AG_X$ is also a generator matrix for $X$.

It follows that
two $[n,k]$-linear codes $X$ and $Y$ with generator matrices
respectively $G_X$ and $G_Y$ are \emph{equivalent} if there
exists $A\in\GL(k,\FF_q)$, a permutation matrix $P\in\GL(n,\FF_q)$
and a diagonal matrix $D\in\GL(n,\FF_q)$ such that
\[ G_X=AG_Y(PD). \]
In particular,
two codes are equivalent if and only if any two of
their generator matrices belong to the same orbit under
the action of the group $\mathrm{GL}(k,\FF_q):(\FF_q^{\times}\wr S_n)$, where $\mathrm{GL}(k,\FF_q)$ acts on the left
of the generator matrix and fixes each code, regarded as a subspace,
while $\FF_q^{\times}\wr S_n$ acts on the right.

The following lemma provides explicit generators for
the group $\cM(V)$ and describe their action on the set of
the linear codes.
\begin{lemma}\label{struttura}
	Let $\fB=({\mathbf e}_1,\dots,{\mathbf e}_n)$ be a fixed basis of $V$ and $\alpha$
	a generator of $\FF_q^{\times}$.
	A set of generators for the monomial group $\cM(V)$ with respect to
	$\fB$ is given by
	all the linear functions $\tau_{ij}:V\to V$ with $1\leq i<j\leq n$
	and $\mu:V\to V$
	such that
	\[ \tau_{ij}({\mathbf e}_k):=\begin{cases}
			{\mathbf e}_i & \text{ if $k=j$ }              \\
			{\mathbf e}_j & \text{ if $k=i$ }              \\
			{\mathbf e}_k & \text{ if $k\not\in\{i,j\}$, }
		\end{cases} \qquad
		\mu({\mathbf e}_k):=\begin{cases}
			\alpha {\mathbf e}_1 & \text{ if $k=1$ }     \\
			{\mathbf e}_k        & \text{ if $k\neq 1$.}
		\end{cases} \]
	In particular, the map $\tau_{ij}$
	is represented with respect to $\fB$ by
	a permutation matrix, while $\mu$ is
	represented by the matrix
	$\mathrm{diag}(\alpha,1,1,\dots,1)$.
\end{lemma}
\begin{proof}
	By construction $\cM(V)$ is generated by the generators
	of the constituent groups of the wreath product $\FF_q^{\times}\wr S_n$.
	As $\FF_q^{\times}$ is cyclic generated by $\alpha$ and
	$S_n$ is generated by the swaps, it follows that the above set
	is enough to generate $\cM(V)$.
\end{proof}

\begin{theorem}
	\label{equiv}
	Take $C\in\cC_t(n,k)$.
	The orbit of $C$ under the action of $\cM(V)$ is contained in
	a connected component of $\Delta_t(n,k)$ for any $t\leq k$.
\end{theorem}
\begin{proof}
	Let $C\in\cC_t(n,k).$ We claim that $C\sim_{\Delta}\rho(C)$, for $\rho\in \cM(V).$
	
	By Lemma~\ref{struttura},
	the group $\cM(V)$ is generated by transpositions $\tau_{ij}$ ($1\leq i<j\leq n$) and the diagonal transformation $\mu$, so it is enough prove the claim above for $\rho=\tau_{ij}$ and $\rho=\mu.$
		
	Suppose $G_C=(P_1,\dots,P_i,\dots,P_n)$ is a given generator matrix of $C$, where
	$P_1,\dots,P_n$ are the columns, hence  vectors of $\FF_{q}^k$. Let $G_{\mu(C)}$ be the generator matrix of the code $\mu(C)$, hence
	$G_{\mu(C)}=G_C\cdot \mathrm{diag}(\alpha,1,1,\dots,1)=(\alpha P_1,\dots, P_i,\dots,P_n)$, since $\mu $ is represented by $\mathrm{diag}(\alpha,1,1,\dots,1)$, with $0\not= \alpha \in \FF_q.$
	
	We study the rank of the matrix $\begin{pmatrix}
			G_C \\ G_{\mu(C)} \end{pmatrix}$. We have
	\begin{multline*}
		\label{told}    \mathrm{rank}
		\begin{pmatrix}
			P_1        & P_2 & \dots & P_i & \dots & P_n \\
			\alpha P_1 & P_2 & \dots & P_i & \dots & P_n \\
		\end{pmatrix}= \\
		\mathrm{rank}
		\begin{pmatrix}
			P_1            & P_2 & \dots & P_i & \dots & P_n \\
			(\alpha-1) P_1 & 0   & \dots & 0 & \dots & 0   \\
		\end{pmatrix}\leq k+1.
	\end{multline*}
	In particular, either
	$\mathrm{rank}\begin{pmatrix}
			G_C \\ G_{\mu(C)}
		\end{pmatrix}=k$ or
	$\mathrm{rank}\begin{pmatrix}
			G_C \\ G_{\mu(C)}
		\end{pmatrix}=k+1$.
	In the former case,
	$C=\mu(C)$ (and there is nothing to prove); in the latter,
	$\dim(C\cap \mu(C))=k-1$ and $C$ and $\mu(C)$
	are adjacent in the graph $\Delta_t(n,k)$.
	
	
	We consider now the action of swaps $\tau_{ij}$, $1\leq i<j\leq n$, on $G_C$. If $P_i=P_j$, then $\tau_{ij}(C)=C$ and there is nothing to
	say.   If $P_i\not= P_j,$ then the code $\tau_{ij}(C)$ has generator matrix
	$G_{\tau_{ij}(C)}=    \begin{pmatrix}
			P_1 & \dots & P_j & \dots & P_i & \dots & P_n \\
		\end{pmatrix}.$
	
	Then
	\begin{multline*}
		\mathrm{rank}\begin{pmatrix}
			G_C \\ G_{\tau_{ij}(C)}
		\end{pmatrix}=
		\mathrm{rank}
		\begin{pmatrix}
			P_1 & \dots & P_i & \dots & P_j & \dots & P_n \\
			P_1 & \dots & P_j & \dots & P_i & \dots & P_n \\
		\end{pmatrix}= \\
		\mathrm{rank}
		\begin{pmatrix}
			P_1 & \dots & P_i     & \dots & P_j     & \dots & P_n \\
			0   & \dots & P_j-P_i & \dots & P_i-P_j & \dots & 0   \\
		\end{pmatrix}= \\
		\mathrm{rank}
		\begin{pmatrix}
			P_1 & \dots & P_i+P_j & \dots & P_j     & \dots & P_n \\
			0   & \dots & 0       & \dots & P_i-P_j & \dots & 0   \\
		\end{pmatrix}\leq k+1.
	\end{multline*}
	Proceeding as before, we have that either $C=\tau_{ij}(C)$ or $C\sim_{\Delta} \tau_{ij}(C).$
      \end{proof}
      It has been pointed out to the authors that
      Proposition 2 of~\cite{Pu} is equivalent to Theorem~\ref{equiv}. 
\section{General results on $\cC_t(n,k)$ and the graph $\Lambda_t(n,k)$}\label{general properties}
Let $C$ be a $[n,k]$-linear code and $G_C= (P_1,P_2,\dots, P_n)$  a given generator matrix of $C$, where $P_1,\dots, P_n$ denote column vectors in $\FF_q^k$.
\begin{lemma}\label{subcodes}
	All $[n,k-1]$-subcodes of $C$ can be represented by generator matrices $G_C^{\varphi}:=(\varphi(P_1), \varphi(P_2), \dots \varphi(P_n)),$ as $\varphi$ varies in all possible ways in the set $\mathcal{M}$
	 of all surjective linear functions from $\FF_q^k$ to $\FF_q^{k-1}.$
\end{lemma}
\begin{proof}
	Let $D$ be an $[n,k-1]$-subcode of $C$. Then a generator matrix $G_D$
	for $D$ has, as rows, $k-1$ linearly independent vectors of $\FF_q^n$,
	and each of these vectors is a linear combination of the rows of $G_C$.
	So, there exists a $(k-1)\times k$ matrix $F$ of rank $k-1$ such that
	$G_D=FG_C$. If $\varphi$ is the linear function $\FF_q^k\to\FF_q^{k-1}$
	with matrix $F$ with respect to the canonical bases of
	$\FF_q^k$ and $\FF_q^{k-1}$, then
	$\varphi$ is surjective. Define $G_C^{\varphi}:=
		(\varphi(P_1),\dots,\varphi(P_n))$ as the matrix having as columns the images under $\varphi$ of the columns $P_1,\dots, P_n.$
	Then $G_C^{\varphi}=G_D.$
	Conversely, if $\varphi:\FF_q^k\to\FF_q^{k-1}$ is a surjective linear function,
	then put $G^{\varphi}_C:=FG_C$ for $F$ a $(k-1)\times k$ matrix (of rank $k-1$) representing $\varphi$.
	Clearly, $\mathrm{rank}(G_C^{\varphi})=k-1$ and all rows of $G_C^{\varphi}$
	are linear combinations of rows of $G_C$; so
	$G_C^{\varphi}$ is the generator matrix of a subcode of $C$.
\end{proof}
\begin{remark}
	Any map $\varphi$ of the previous lemma can be regarded as a function
	which sends all $k$-dimensional subspaces $C$ of $V$ into $(k-1)$-dimensional
	subspaces $D$ with $D\subset C$. Its explicit action depends on
	the choice of a basis for each $k$-dimensional subspace $C$ of $V$, i.e.
	on the choice of a generator matrix $G_C$ for $C$.
	In particular, $\varphi$ can be regarded as a map $\FF_q^{k,n}\to\FF_q^{k-1,n}$
	sending a $k\times n$ matrix $G$ into a $(k-1)\times n$ matrix
	$G'$, $G\mapsto G'=FG,$ where $F$ is the representative matrix of
	$\varphi$ (with respect to given bases).
	In any case, by Lemma~\ref{subcodes},
	the set of all subcodes of $C$ does not  depend on the specific choice of $G_C$ for $C$ since   $\varphi$
	varies among all surjective linear functions $\FF_q^k\to\FF_q^{k-1}$.
\end{remark}

\begin{lemma}
	\label{e2ll}
	Let $C\in\cC_t(n,k)\setminus\cI_t(n,k)$ and $C'$ be equivalent $[n,k]$-codes. Then there exist equivalent $[n,k-1]$-subcodes $D, D'\in\cC_t(n,k-1)$ with $D\subset C$ and $D'\subset C'$ respectively.
        In particular,
        $C'\in\cC_t(n,k)\setminus\cI_t(n,k)$.
\end{lemma}
\begin{proof}
  Since $C$ and $C'$
  are equivalent $[n,k]$-codes, then $C'\in\cC_t(n,k)$, and
  there exists $\rho\in \cM(V)$ such that $\rho(C)=C'$.
	Take $D\in\cC_t(n,k-1)$ as a subcode of $C$ ($D$ exists because $C$ is not isolated) and
	let $G_C$  be a generator matrix for $C$.
	As seen in Section~\ref{Sec2}, $G_{\rho(C)} (=\rho(G_{C}))$ is a generator matrix of $C'.$
	
	Denote by $\varphi:\FF_q^k\to\FF_q^{k-1}$ the surjective map associated to $D$ (see Lemma~\ref{subcodes}), so
	$G_C^{\varphi}$ is a generator matrix of $D$.
	We claim that $D$ is equivalent to the subcode of $\rho(C)$ determined by $\varphi$, i.e.
	\begin{equation}
		\label{eeq}
		G_{\rho(C)}^{\varphi}=\rho(G_C^{\varphi}).
	\end{equation}
	By Lemma~\ref{struttura}, it is enough
	to prove Claim~(\ref{eeq}) for $\rho=\tau_{ij}$ and $\rho=\mu.$
	
	Write $G_C=(P_1,\dots,P_i,\dots,P_j,\dots,P_n)$.
	If $\tau_{ij}$ acts as a transposition on the columns $P_i$ and $P_j$ of $G_C$, then
	\begin{multline*}  \tau_{ij}(G_C^{\varphi})=( \tau_{ij}\circ\varphi)(G_C)=
		(\tau_{ij}\circ\varphi)((P_1,\dots,P_i,\dots,P_j,\dots P_n))=\\
		\tau_{ij}(\varphi(P_1),\dots,\varphi(P_i),\dots,\varphi(P_j),\dots,\varphi(P_n))=\\
		(\varphi(P_1),\dots,\varphi(P_j),\dots,\varphi(P_i),\dots,\varphi(P_n))=   \varphi((P_1,\dots,P_j,\dots,P_i,\dots,P_n))=\\
		\varphi(\tau_{ij}(P_1,\dots,P_i,\dots,P_j,
			\dots P_n))=  \varphi(G_{\tau_{ij}(C)})=  G_{\tau_{ij}(C)}^{\varphi}.
	\end{multline*}
	
	Likewise, if $\mu$ acts by multiplying the column $P_1$ by a primitive
	element $\alpha$ of $\FF_q^{\times}$, then
	\begin{multline*}
		\mu(G_C^{\varphi})=( \mu\circ\varphi)(G_C)=
		\mu(\varphi(P_1),\varphi(P_2),\dots,\varphi(P_n))= \\
		(\alpha\varphi(P_1),\varphi(P_2),\dots,\varphi(P_n))=
		(\varphi(\alpha P_1),\varphi(P_2)\dots,\varphi(P_n))= \\
		\varphi((\alpha P_1,P_2,\dots,P_n))=
		\varphi(\mu(P_1,P_2,\dots,P_n))= \varphi (G_{\mu(C)})=G_{\mu(C)}^{\varphi}.
	\end{multline*}
	Thus, the Claim~(\ref{eeq}) is proved. In particular,
        $D'=\rho(D)\in\cC_t(n,k-1)$ and $D'\subseteq C'$ so
        $C'\in\cC_t(n,k)\setminus\cI_t(n,k)$.
\end{proof}

\begin{lemma}
	\label{cc}
	Let  $C,C'\in\cC_t(n,k)\setminus\cI_t(n,k)$ and  suppose that $D$ and $D'$ are two $[n,k-1]$-linear subcodes of $C$ and $C'$ ,respectively. If  $D,D'$ are  in the same
	connected component of $\Delta_t(n,k-1)$, then $C$ and $C'$ are in the same connected component   of $\Lambda_t(n,k)$.
\end{lemma}
\begin{proof}
	By hypothesis, there exists a path $D_0=D\sim_{\Delta} D_1\sim_{\Delta}\dots\sim_{\Delta}
		D_w=D'$ in
	$\Delta_t(n,k-1)$ from $D$ to $D'$.
	For any $i=0,\dots,w-1$ we
	have $\dim(D_i\cap D_{i+1})=k-2$ and
        $\dim(\langle D_i,D_{i+1}\rangle)=k.$
	Consider the codes $E_i:=\langle D_{i},D_{i+1}\rangle\in\cC_t(n,k)$.
	Since $E_i\cap E_{i+1}=D_{i+1}\in\cC_t(n,k-1)$
	for all $i$, they are the vertices
	of a path $E_0\sim_{\Lambda} E_1\sim_{\Lambda}\dots\sim_{\Lambda} E_{w-1}$ in $\Lambda_t(n,k)$.
	Furthermore, $C$ is adjacent to
	$E_0$ in $\Lambda_t(n,k)$ since $D\subseteq C\cap E_0$ and
	likewise $C'$ is adjacent to $E_{w-1}$. It follows that $C$ and
	$C'$ are in the same connected component of $\Lambda_t(n,k)$.
\end{proof}

Recall that
$\widetilde{\Lambda}_t(n,k)$ is the graph induced by $\Lambda_t(n,k)$
on the codes in $\cC_t(n,k)\setminus\cI_t(n,k).$
The following is immediate.
\begin{corollary}
	\label{connected}
	Suppose that the graph $\Delta_t(n,k-1)$ is connected.
	Then the following conditions hold true:
	\begin{enumerate}
		\item All elements of $\cC_t(n,k)\setminus\cI_t(n,k)$
		      belong to the same connected component $\widetilde{\Lambda}_t(n,k)$
		      of $\Lambda_t(n,k)$. In particular, if
		      the graph $\Delta_t(n,k-1)$ is connected, then $\widetilde{\Lambda}_t(n,k)$ is connected.
		\item
		      Any two codes in $\cC_t(n,k)\setminus\cI_t(n,k)$
		      are at distance at most
		      $\mathrm{diam}(\Delta_t(n,k-1))+1$.
	\end{enumerate}
\end{corollary}

\begin{corollary}
	Suppose $C,C'\in\cC_t(n,k)\setminus\cI_t(n,k)$ are equivalent codes.
	Then $C$ and $C'$ belong to the same connected component of
	$\Lambda_t(n,k)$.
\end{corollary}
\begin{proof}
	By Lemma~\ref{e2ll}, there are equivalent subcodes $D,D'$ of $C$ and $C'$
	(respectively) with $D,D'\in\cC_t(n,k-1)$.
	By Theorem~\ref{equiv}, $D$ and $D'$ belong to the same connected
	component of $\Delta_t(n,k-1)$.
	The thesis now follows directly from Lemma~\ref{cc}.
\end{proof}

In Lemma~\ref{l:nc} we shall make use of the following remark.
\begin{remark}
	\label{r:above}
	For any $D\in\cC_t(n,k-1)$ and for any $[n,k]$-linear code $C$ with $D\subset C$ we have
	$C\in\cC_t(n,k)$; consequently, for any element $D\in\cC_t(n,k-1)$
	with $k-1<n$ there exists at least one $C\in\cC_t(n,k)$ with $D\subset C$.
\end{remark}

\begin{lemma}
	\label{l:nc}
	Suppose $\Delta_{t}(n,k-1)$ not to be connected.
	Then the graph $\widetilde{\Lambda}_t(n,k)$ is not connected.
\end{lemma}
\begin{proof}
	Take two codes $D,D'\in\cC_t(n,k-1)$ which belong to different
	connected components of $\Delta_t(n,k-1)$ and
	let $C,C'$ be two
	codes in $\cC_t(n,k)$ having $D$, $D'$ as subcodes respectively.
	These codes exist by Remark~\ref{r:above}. By contradiction, suppose there exists a path $C=C_0\sim_{\Lambda} C_1\sim_{\Lambda}\dots
		\sim_{\Lambda} C_w=C'$ in $\widetilde{\Lambda}_t(n,k)$ joining $C$ and $C'.$ So,
	$D_i=C_i\cap C_{i+1}\in\cC_t(n,k-1)$ for $i=0,\dots,w-1$.
	Since both $D_i$ and $D_{i+1}$ are contained in $C_{i+1},$
	we have either $D_i=D_{i+1}$ or $\dim(D_i\cap D_{i+1})=k-2.$
	In either case, we have a collection of subcodes
	$D=D_{i_0}\sim_{\Delta} D_{i_1}\sim_{\Delta}\dots\sim_{\Delta} D_{i_r}=D'$
	in $\cC_t(n,k-1)$ such that
	$\dim(D_{i_j}\cap D_{i_{j+1}})=k-2$, i.e. a path in
	$\Delta_t(n,k-1)$. This is a contradiction, since we assumed
	$D$ and $D'$ to be in distinct connected components of
	$\Delta_t(n,k-1)$.
\end{proof}

From Corollary~\ref{connected} and Lemma~\ref{l:nc}, we have the following.
\begin{corollary}
	\label{co:con}
	The graph $\widetilde{\Lambda}_t(n,k)$ is connected if and
	only if $\Delta_t(n,k-1)$ is connected.
\end{corollary}

\section{Proof of the main theorems} \label{sec3}
As in Lemma~\ref{subcodes}, put  \[\mathcal{M}=\{\varphi \colon \FF_q^k\rightarrow \FF_q^{k-1}\colon \,\,\varphi\text{ is linear and } \mathrm{Im}(\varphi)=\FF_q^{k-1}\} \,\,\text{ and }\]
\[\mathcal{K}=\{\ker(\varphi)\colon \varphi\in \mathcal{M}\}.\] 
Clearly, $\mathcal{K}\cong \PG(k-1,q)$, hence $|\mathcal{K}|=(q^k-1)/(q-1).$
\\

\noindent {\bf Proof of Theorem~\ref{main thm 1}}.
Since $\Lambda_t(n,k)=\widetilde{\Lambda}_t(n,k)\cup\cI_t(n,k)$ and
the elements of $\cI_t(n,k)$ are isolated vertices of
$\Lambda_t(n,k)$,
Theorem~\ref{main thm 1} is a direct consequence of Corollary~\ref{co:con}.\hfill $\square$
\\

\noindent {\bf Proof of Theorem~\ref{main thm 2}}
We have that $C\in \cC_t(n,k)$ is an isolated code if and only if every $[n,k-1]$-subcode of $C$ does not belong to $\cC_t(n,k-1)$.
This is the same as to say that
for every $[n,k-1]$-subcode $D$ of $C$ there exist at least $t$ columns in any of its
generator matrices which are linearly dependent.
By Lemma~\ref{subcodes}, any subcode $D$
has a generator matrix of the form $G^{\varphi}_{C}=(\varphi(P_1), \dots, \varphi(P_n))$, where $G_C=(P_1,\dots, P_n)$ is a given generator matrix for $C$
with columns $P_1,\dots,P_n$
and  $\varphi$ is an element of $\cM$. For $C$ to be isolated
we require that
for every $\varphi \in \cM$ there is at least one set of $t$ columns of $G_C^{\varphi}$  which are linearly dependent.
This is equivalent to saying that there exist $i_1,\dots  i_t\in\{1,\dots, n\}$ such that $\lambda_{i_1}P_{i_1}+\dots \lambda_{i_t}P_{i_t}\in \ker(\varphi)$,
i.e. the projective point given by $\ker(\varphi)(\in \mathcal{K})$ is in the $(t-1)$-dimensional projective space of $\PG(k-1,q)$ spanned by the points $\langle P_{i_1} \rangle,\dots,  \langle P_{i_t} \rangle$.
Since this needs to be true for all $\varphi\in\cM$ we have that the columns
of the generator matrix $G_C$ of $C$ must form a $(t-1)$-saturating set of $\PG(k-1,q)$.
Theorem~\ref{main thm 2} follows. \hfill {$\square$}

\subsection{Non-degenerate codes}\label{non-deg codes}
By Theorem~\ref{main thm 2},
we have that $C\in \cC_1(n,k)$, with generator matrix $G_C=(P_1,\dots, P_n)$, is an isolated code if and only if
for every $\varphi \in \mathcal{M}$ there exists $i\in\{1,\dots, n\}$ such that $\langle P_i\rangle =\ker(\varphi)$, i.e. $\forall v\in \mathcal{K}$ there exists $i\in\{1,\dots, n\}$ such that $\langle P_i\rangle=v$.

The following lemma is a direct consequence of Theorem~\ref{told} and Theorem~\ref{main thm 1}. We also propose an alternative direct proof which holds only for $t=1$.
\begin{lemma}~\label{2 main thm1}
	$\Lambda_1(n,k)=\cI_1(n,k)\cup \widetilde{\Lambda}_1(n,k)$, where $\widetilde{\Lambda}_1(n,k)$ is a
	connected component of $\Lambda_1(n,k)$.
        The diameter of $\widetilde{\Lambda}_1(n,k)$ is at most $k+1.$
\end{lemma}
\begin{proof}
	It is enough to show that if a code is not isolated, then it belongs
	to the   same connected component as any other non-isolated code.
	First observe that a code which contains the all-$1$ vector $\mathbf{1}$,
	is not isolated.
	Let now $C_1$ and $C_2$ be two non-isolated codes and suppose that
	$D_1$ and $D_2$ are two non-degenerate $[n,k-1]$-subcodes of
	$C_1$ and $C_2$, respectively.
	If $\mathbf{1}\in C_i$, then put $C_i':=C_i$ for $i=1,2$.
	If $\mathbf{1}\not\in C_i$, then put $C_i'=\langle D_i,\mathbf{1}\rangle$.
	Observe that $d(C_i,C_i')\leq1$ in $\Lambda_1(n,k)$.
	Since $\Delta_1(n,k-1)$ is connected (Theorem~\ref{told}) we can construct a path in $\Lambda_1(n,k)$ from
	$C_1'$ to $C_2'$ of codes all
	containing the vector $\mathbf{1}.$
	So, there is a path from $C_1$ to $C_1'$ and then from $C_1'$ to $C_2'$ and
	finally to $C_2'$ to $C_2$, i.e. $C_1$ and $C_2$ belong to the same connected
	component.
	Furthermore, observe that $d(C_1,C_2)\leq d(C_1,C_1')+d(C_1',C_2')+d(C_2',C_2)
		\leq 2+d(C_1',C_2')\leq 2+k-1=k+1$.
\end{proof}

\begin{lemma}\label{1 main thm1}
	$\Lambda_1(n,k)$ is connected if and only if $n< (q^k-1)/(q-1).$ 
\end{lemma}
\begin{proof}
	By Lemma~\ref{2 main thm1}, the graph  $\Lambda_1(n,k)$
	is connected if and only if there exists no isolated vertex in it,
	i.e. $\cI_{1}(n,k)=\emptyset$.
	This is equivalent to saying that for every $C\in \cC_1(n,k)$, there exists at least a $[n,k-1]$-subcode $D$ which is non-degenerate.
	By Lemma~\ref{subcodes}, we have that  $C$ admits at least a non-degenerate subcode if and only if there is at least one
	function in $\cM$ whose kernel does not contain any column of $G_C.$
	This is always guaranteed if
	the number of surjective linear functions from $\FF_q^k $ onto  $\FF_q^{k-1}$, and hence the number of their kernels, is strictly greater than  the number of  columns of a generator matrix of $C$ (i.e. the length of $C$), hence for $|\mathcal{K}|=\frac{q^k-1}{q-1}>n.$
	On the other hand, if $n=\frac{q^k-1}{q-1}$ the code whose generator matrix has a column set consisting of all the distinct vector representatives of the
	$1$-dimensional subspaces of $\FF_q^k$ is definitely in $\cI_{1}(n,k)$
	and so $\Lambda_1(n,k)$ is disconnected.
\end{proof}

Corollary~\ref{t=1} follows from Lemma~\ref{2 main thm1} and Lemma~\ref{1 main thm1} or, alternatively, Theorem~\ref{main thm 2}.

\subsection{Projective codes}\label{projective codes}
We first prove that the Grassmann graph of the projective codes $\Delta_2(n,k)$ is always connected.
Observe that, in any case,
there is a natural bound on $n$, that is $n\leq\frac{q^k-1}{q-1}$ as the
maximum length of a projective code of dimension $k$ is the same as
the number of points of $\PG(k-1,q)$.

\begin{lemma} \label{connessione Pankov t=2}
	The graph $\Delta_2(n,k)$ is connected for any $q$ and any $2\leq k\leq n$.
      \end{lemma}
      For a different proof of this lemma see also~\cite[Theorem 1]{Pu}.
\begin{proof}
	Let $C$ and $C'$ be two projective codes with generator matrices respectively
	$G_C=(P_1, P_2, \dots, P_n)$
	and
	$G_{C'}=(P_1', P_2', \dots, P_n')$, where $P_i,P_i'\in\FF_q^k$ are column
	vectors.
	By Theorem~\ref{equiv}, we can assume, up to code equivalence,
	that if a column $P_i$ in $G_C$
	is proportional to a column $P_j'$ in $G_{C'}$, then $i=j$.
        Also by Theorem~\ref{equiv}, we can assume that if $P_i$ in $G_C$ is
        proportional to $P_i'$ in $G_{C'}$, then $P_i=P_i'$ (i.e. the
        coefficient of proportionality is $1$ and the columns are the same).
        If under these assumptions $G_C=G_{C'}$, then $C$ and $C'$ are
        equivalent codes and, by Theorem~\ref{equiv} they belong to the
        same connected component of $\Delta_2(n,k)$.
	Otherwise we may
        construct a path in $\Delta_2(n,k)$ joining $C$ and $C'$ by
	considering the codes
        $C_0:=C\sim_{\Delta} C_1\sim_{\Delta}\dots\sim_{\Delta} C_w=:C'$,
        where the code
	$C_i$ has as generator matrix, the matrix  where
        the first $i$ columns are the
	first $i$ columns of $G_{C'}$ and
	the remaining columns are the corresponding columns of $G_C$. Indeed, $\mathrm{rank}\begin{pmatrix}
			G_{C_i} \\ G_{C_{i+1}}
		\end{pmatrix}=k$ or $\mathrm{rank}\begin{pmatrix}
			G_{C_i} \\ G_{C_{i+1}}
		\end{pmatrix}=k+1$ so, either $C_i=C_{i+1}$ or $C_i\sim_{\Delta} C_{i+1}.$
\end{proof}

Part 1. of Corollary~\ref{t=2} is a direct consequence of Lemma~\ref{connessione Pankov t=2} and Theorem~\ref{main thm 1}. The following is part 2. of Corollary~\ref{t=2}. Part 3. of Corollary~\ref{t=2} is Theorem~\ref{main thm 2}.
\begin{lemma}
	$\Lambda_2(n,k)$ is connected if and only if $n$ is strictly less than the
	minimum size $\mu_1$ of a $1$-saturating set for $\PG(k-1,q)$.
\end{lemma}
\begin{proof}
	By Corollary~\ref{t=2} part 1., the graph  $\Lambda_2(n,k)$
	is connected if and only if there exists no isolated vertex in it.
	This is equivalent to saying that for every $C\in \cC_2(n,k)$, there exists at least a $[n,k-1]$-subcode $D$ which is projective.
	By Lemma~\ref{subcodes}, we have that  $C$ admits at least a
	projective subcode if and only if there is at least one
	function in $\cM$ whose kernel does not belong to any $2$-space
	spanned  by $2$ columns of $G_C.$
	This is always guaranteed if
	the number of  columns of any generator matrix
	of $C$ (i.e. the length of $C$) is strictly less than the
	minimum size of a $1$-saturating set of $\PG(k-1,q)$.
	
	On the other hand, if $n\geq\mu_1$, where $\mu_1$ is the
	minimum size of a $1$-saturating set $\Omega$ of $\PG(k-1,q)$, the code whose generator matrix contains as columns
	vector representatives for all of the
	$1$-dimensional subspaces of $\Omega$ is definitely in $\cI_{2}(n,k)$
	and so $\Lambda_2(n,k)$ is disconnected.
\end{proof}

\subsection{Proof of Theorem~\ref{large-q}} \label{general codes}
\begin{lemma}
	\label{l-emp}
	If $t=k$, then $\cC_{k}(n,k)=\cI_{k}(n,k)$.
	If $t<k$
	and $q>{n\choose t}^{1/(k-t)}$, then
	we have $\cC_t(n,k)\neq\emptyset$
	and $\cI_t(n,k)=\emptyset$.
\end{lemma}
\begin{proof}
	If $k=t$, then  each set of $k$ columns of any generator matrix of a code
	$C\in\cC_{k}(n,k)$ is linearly independent, so $C$ is MDS.
	In particular, the dual minimum distance of each $[n,k-1]$-subcode of an MDS
	code is at most $k-1<k$ (as its generator matrix has just $k-1$ rows).
	So, every code $C\in\cC_{k}(n,k)$ is isolated and $\cC_{k}(n,k)=\cI_{k}(n,k)$.
	
	Suppose now $t\leq k-1$ and $q>n$.
	Take $C\in\cC_t(n,k)$ with generator matrix $G_C$. Let $\varphi:\FF_q^k\to\FF_q^{k-1}$ be
	a surjective linear map and denote by $D=\varphi(C)$
	the subcode of $C$ having as  generator matrix $G_C^{\varphi}.$
	We have $D\not\in\cC_t(n,k-1)$ if and only if there is
	a set $\Omega$ of $t$ linearly independent
	columns in $G_C$  which is mapped into
	a set of linearly dependent columns in $G_C^{\varphi}$. In other words, this
	happens if and only if $\ker(\varphi)\in\langle\Omega\rangle$.
	As $\varphi$ varies in all possible ways, there are exactly $(q^k-1)/(q-1)$
	possible subspaces for $\ker(\varphi).$ On the other hand the  number of
	possible $t$-subspaces spanned by $t$ columns of $G_C$ is
	${n\choose {t}}$ and each of these contains at most
	$(q^{t}-1)/(q-1)$ distinct $1$-dimensional subspaces.
	So, if
	\[ (q^k-1)>{n\choose t}(q^{t}-1),
        \]
        then
	we have that there is at least one $\varphi\in \cM$ such that the subcode of $C$ determined by $\varphi$ is in
	$\cC_{t}(n,k-1)$ and, consequently, $C\not\in\cI_t(n,k)$.
	So, if
	\[ \frac{q^k-1}{q^t-1}>{n\choose t}, \]
	then $\cI_{t}(n,k)=\emptyset$.
	
	For $t<k$ we have the approximation
	$q^{k-t}<\frac{q^k-1}{q^t-1}$,
	since
	\[ q^{k-t}=\frac{q^k}{q^t}<\frac{q^k-1}{q^t-1}
		\Leftrightarrow
		q^{k}(q^t-1)<q^t(q^k-1)
		\Leftrightarrow
		q^k>q^t \Leftrightarrow k>t. \]
	So, if we require $q^{k-t}> {n\choose t}$, then $\frac{q^{k}-1}{q^{t}-1} > {n\choose t}$. Hence $\cI_t(n,k)=\emptyset.$
\end{proof}
Theorem~\ref{large-q} now follows from Theorem~\ref{told}, Theorem~\ref{main thm 1} and
Lemma~\ref{l-emp}.

\subsection{Transparent embeddings}\label{transparenza}
Given two point-line geometries $\cG_1$ and $\cG_2$, we say that an injective map
$\varepsilon: \cG_1\to \cG_2$ of the point-set of $\cG_1$ into the point-set
of $\cG_2$
is an \emph{embedding} of $\cG_1$
into $\cG_2$ if
for any line $\ell$ of $\cG_1$ the image $\varepsilon(\ell):=\{
	\varepsilon(p)\colon p\in\ell\}$ is a line of $\cG_2$.
If $\cG_2$ is a projective geometry we say that an embedding of $\cG_1$ into $\cG_2$ is a projective embedding into the projective subspace $\langle\varepsilon(\cG_1)\rangle$.

In~\cite{transp} the notion of \emph{transparency} for an
embedding has been introduced. An embedding $\varepsilon:\cG_1\to \cG_2$ is transparent if the preimage of any line of $\cG_2$
contained in the image of $\varepsilon$
is a line of $\cG_1$.

Let us focus now on the Grassmann geometry $\cG(n,k)$ and the geometry $\cP_t(n,k).$ Since $\cP_t(n,k)$ is a subgeometry of $\cG(n,k),$ the inclusion map $\iota:\cP_t(n,k)\to\cG(n,k)$ is an embedding.

There is a vast literature regarding projective embeddings of the Grassmann geometry. For the sake of our paper we stick only to the essential results  referring the interested reader to, e.g.,~\cite{HIG}.
\bigskip

\noindent {\bf Proof of Theorem~\ref{t0}}.
By definition of $\cP_t(n,k)$, any line of $\cP_t(n,k)$ is also
a line of $\cG(n,k).$ In particular, the inclusion $\iota$
is an embedding $\iota:\cP_t(n,k)\to\cG(n,k)$.

To prove that $\iota$ is transparent we need to show that if
a line $\ell$ of $\cG(n,k)$ consists all of points of $\cC_t(n,k)$, then
for any two $X,Y\in\ell$ we have $X\sim_\Lambda Y$. This is equivalent
to saying that if
$X$ and $Y$ are not on a line of $\Lambda_t(n,k),$ then $\iota(X)$ and $\iota(Y)$ are not on a line contained in $\iota(\cC_t(n,k))$, i.e.
either $X$ and $Y$ are not on a line of $\cG(n,k)$ or
$X$ and $Y$ are on a line of $\cG(n,k)$ whose points are not all images of points in $\cC_t(n,k)$.
If $\iota(X)$ and $\iota(Y)$ are not collinear in $\cG(n,k)$, then
$X$ and $Y$ are not collinear in $\cP_t(n,k)$ as all lines of
$\cP_t(n,k)$ are lines of $\cG(n,k)$. So, we have to deal only with
the latter case. Observe that $X$ and $Y$ lie on a line of
$\cG(n,k)$ if and only if $X\sim_{\Delta}Y$.

So, suppose $X\sim_{\Delta}Y$ but $X\not\sim_{\Lambda}Y$ for
$X,Y\in\cC_t(n,k)$. To obtain the result
it is enough to show that there exists at least one $Z$
on the line $\ell$
of $\cG(n,k)$ determined by $X$ and $Y$ with $Z\not\in\cC_t(n,k)$.
So, the theorem is a consequence of the following lemma.
\begin{lemma}
	Suppose $X,Y\in\cC_t(n,k)$, $\dim(X\cap Y)=k-1$ but
	$X\cap Y\not\in\cC_t(n,k-1)$. Then there exists $Z$ such that
	$X\cap Y\subset Z \subset \langle X,Y\rangle$, $\dim(Z)=k$ and $Z\not\in\cC_t(n,k)$.
\end{lemma}
\begin{proof}
	Put $D=X\cap Y$. By hypothesis, $D\not\in\cC_t(n,k-1)$. Suppose
	by contradiction that for all $\mathbf{v}\in\langle X, Y\rangle\setminus D$ we have
	$Z_{\mathbf v}:=\langle D,\mathbf{v}\rangle\in\cC_t(n,k)$.
	Write $X=\langle D, \mathbf{x}\rangle $ and $Y=\langle D, \mathbf{y}\rangle $, where $\mathbf{x}=(x_1,\dots,x_n)\in X\setminus D$ and
	$\mathbf{y}=(y_1,\dots,y_n)\in Y\setminus D$. So,
	we can always
	take $\mathbf{v}=\alpha\mathbf{x}+\beta\mathbf{y}$, for $\alpha$, $\beta\in \FF_q.$

	Since $D\not\in\cC_t(n,k-1)$, there is at least one set of $t$ columns $P_{i_1},\dots,P_{i_t}$ with $1\leq i_1<\dots<i_t\leq n$
	in any generator matrix $G_D$ of $D$ which are linearly dependent.
	Denote by $\overline{G_{D}}$	the $(k-1)\times t$ submatrix of $G_D$ having $P_{i_1},\dots,P_{i_t}$ as columns (and $R_1,\dots, R_{k-1}$ as rows). By construction, we
	have that $\mathrm{rank}(\overline{G_{D}})\leq t-1.$
	
	Consider the $(k\times t)$-matrix
	\begin{equation}\label{sottomatrice}
		G_{\alpha,\beta}:=\begin{pmatrix}
			P_{i_1}                      & P_{i_2}                      & \dots & P_{i_t}                      \\
			\alpha x_{i_1}+\beta y_{i_1} & \alpha x_{i_2}+\beta y_{i_2} & \dots & \alpha x_{i_t}+\beta y_{i_t}
		\end{pmatrix}.
	\end{equation}
	
	The matrix $G_{\alpha,\beta}$ can be regarded as the submatrix of a generator matrix of the generic code $Z_{\alpha\mathbf{x}+\beta\mathbf{y}}$ with $D\subset Z_{\alpha\mathbf{x}+\beta\mathbf{y}}\subset\langle X,Y\rangle$.
	If  $Z_{\alpha\mathbf{x}+\beta\mathbf{y}}\in\cC_t(n,k)$, then
	necessarily $\mathrm{rank}(G_{\alpha,\beta})=t$.
	
	On the other hand, since $\begin{pmatrix}
			P_{i_1} & P_{i_2} & \dots & P_{i_t} \\
			x_{i_1} & x_{i_2} & \dots & x_{i_t}
		\end{pmatrix}$
	is a $(k\times t)$-submatrix of a generator matrix of the code $X$ and since $X\in \cC_t(n,k)$, we have that $\mathrm{rank}\begin{pmatrix}
			P_{i_1} & P_{i_2} & \dots & P_{i_t} \\
			x_{i_1} & x_{i_2} & \dots & x_{i_t}
		\end{pmatrix}=t.
	$
	
	So, \[\mathrm{rank}\begin{pmatrix}
			P_{i_1} & P_{i_2} & \dots & P_{i_t} \\
			x_{i_1} & x_{i_2} & \dots & x_{i_t} \\
			y_{i_1} & y_{i_2} & \dots & y_{i_t}
		\end{pmatrix}=t=
		\mathrm{rank}\begin{pmatrix}
			P_{i_1} & P_{i_2} & \dots & P_{i_t} \\
			x_{i_1} & x_{i_2} & \dots & x_{i_t}
		\end{pmatrix}.
	\]
	Thus,
	the vector $(y_{i_1},\dots,y_{i_t})$ is a linear combination
	of the rows $R_1,\dots R_{k-1}$ of $\overline{G_{D}}$ and of the vector
	$(x_{i_1},\dots,x_{i_t})$, where $(x_{i_1},\dots,x_{i_t})$
	appears with a non-zero coefficient (otherwise
	$(y_{i_1},\dots,y_{i_t})\in \langle R_1,\dots,R_{k-1}\rangle$ and
	consequently the columns $\begin{pmatrix}
			P_{i_1} \\
			y_{i_1}
		\end{pmatrix}, \begin{pmatrix}
			P_{i_2} \\
			y_{i_2}
		\end{pmatrix}, \dots, \begin{pmatrix}
			P_{i_t} \\
			y_{i_t}
		\end{pmatrix}$ of the generator
	matrix $\begin{pmatrix}
			G_D \\
			\mathbf{y}
		\end{pmatrix}$ of $Y$ would be linearly dependent and thus
	$Y\not\in\cC_t(n,k)$, against the hypothesis).
	Hence, we can write
	$\lambda (y_{i_1},\dots,y_{i_t})=(x_{i_1},\dots,x_{i_t})+\sum_{j=1}^{k-1}\theta_j R_j$ for suitable $\theta_j, \lambda\in \FF_q $ and
	$\lambda\neq 0.$
	So,
	\[ \lambda (y_{i_1},\dots,y_{i_t})-(x_{i_1},\dots,x_{i_t})=
		\sum_{j=1}^{k-1}\theta_j R_j, \]
	and the rank of the matrix $G_{-1,\lambda}$ (obtained from~(\ref{sottomatrice}) with $\alpha=-1$ and $\beta=\lambda$) is at most $t-1$.
	It follows that $Z_{-\mathbf{x}+\lambda\mathbf{y}}\not\in\cC_t(n,k)$.
\end{proof}

\section*{Acknowledgments}
Both authors are affiliated with GNSAGA of INdAM (Italy) whose support they kindly acknowledge.

\vskip.2cm\noindent
\begin{minipage}[t]{\textwidth}
	Authors' addresses:
	\vskip.4cm\noindent\nobreak
	\begin{minipage}[t]{7cm}
		\small{Ilaria Cardinali\\
			Dep. Information Engineering and \\ Mathematics \\University of Siena\\
			Via Roma 56, I-53100 Siena, Italy\\
			ilaria.cardinali@unisi.it}
	\end{minipage}
	\qquad\qquad
	\begin{minipage}[t]{7cm}
		\small{Luca Giuzzi\\
			D.I.C.A.T.A.M. \\
			University of Brescia\\
			Via Branze 43, I-25123 Brescia, Italy \\
			luca.giuzzi@unibs.it}
	\end{minipage}
\end{minipage}

\end{document}